\documentclass[12pt, english]{article}
\usepackage[T2A]{fontenc}
\usepackage[cp1251]{inputenc}
\usepackage{babel}

\usepackage[centertags]{amsmath}
\usepackage{amsfonts}
\usepackage{amssymb}
\usepackage{amsthm}
\usepackage{newlfont}
\usepackage[active]{srcltx}
\usepackage{bm}

\newtheorem{theorem}{\bf Theorem}
\newtheorem{lemma}{\bf Lemma}
\newtheorem{corollary}{\bf Corollary}

\begin{document}

\title{On strong forms of the Borel--Cantelli lemma and intermittent interval maps
}
\author{Andrei N. Frolov \footnote{This research is supported by RFBR, project 18--01--00393}
\\ Dept. of Mathematics and Mechanics
\\ St.~Petersburg State University
\\ St. Petersburg, Russia
\\ E-mail address: Andrei.Frolov@pobox.spbu.ru}

\maketitle


{\abstract{ 
We derive new variants of the quantitative Borel--Cantelli lemma and
apply them to analysis of statistical properties for some
dynamical systems. We consider intermittent maps of $(0,1]$
which have absolutely continuous invariant probability measures.
In particular, we prove that every sequence of intervals 
with left endpoints uniformly separated from zero is 
the strong Borel--Cantelli sequence
with respect to such map and invariant measure.
} 
}



\section*{1. Introduction and results}

Let $ (\Omega,\mathcal{F},P) $ be a probability space and $ \{A_n\} $ be a sequence
of events. Put
\begin{eqnarray}\label{snen}
S_n = \sum\limits_{k=1}^n \mathbb{I}(A_k),\quad
E_n = \sum\limits_{k=1}^n P(A_k),
\end{eqnarray}
where $\mathbb{I}(\cdot)$ is the indicator of the event in brackets.

The Borel--Cantelli lemma deals with the probability 
$P_\infty=P(\lim\limits_{n\to\infty} S_n=\infty )$. By its first part,
$P_\infty=0$ when $\lim\limits_{n\to\infty} E_n <\infty$. 
If $\lim\limits_{n\to\infty} E_n =\infty$,
then the situation is complicated. The probability $P_\infty$ can be
any number in $[0,1]$. It is known various conditions sufficient for $P_\infty=1$. 
In statements of the second part of the Borel--Cantelli lemma,
one can usually see the condition of a pair-wise independence for 
events under consideration. Unfortunately, this condition fails in many interesting cases.
In further generalizations, one can find
various conditions sufficient for $P_\infty\geqslant L$,
where $L$ is some numerical characteristic generated by the sequence of events.
(Cf. [1--7], for example,  and the references therein.)
One then need some conditions on $\{A_n\}$ which yield $L=1$.

In various applications, strong laws of large numbers for $S_n$ 
are of essential interest. We consider $S_n$ centered at mean $E_n$
and normalized by $f(E_n)$. 
Such strong law of large numbers for $S_n$ is called the quantitative Borel--Cantelli lemma.
Note that $S_n$ is the number of those events from $A_1,\dots,A_n$ which occur. Hence,
one deals with the maximal generalization of the Bernoulli trails when trails can be
dependent and probability of head can change. Khintchine's law of the iterated logarithm
shows a kind of the function $f(x)$ for the Bernoulli case. 
Similar functions will be used below.

In this paper, we derive new variants of the quantitative Borel--Cantelli lemma
and apply them to describe statistical properties of 
some non-uniformly hyperbolic (expanding) dynamical systems.

Our first result is as follows.

\begin{theorem}\label{QBCL}
Let $ \psi(x) $, $x \geqslant 0$, be a non-decreasing positive function with 
$ \sum\limits_{n=1}^\infty 1/(n \psi(n))<\infty $ and 
$g(x)$, $x \geqslant 0$, be a positive function such that 
$ g(x)/x $  and $ x^{2-\delta}/g(x) $ are non-decreasing for some $\delta\in (0,1)$.
Assume that $ E_n \to\infty $ as $ n\to\infty $ and 
\begin{eqnarray}\label{var}
 Var (S_m-S_n) \leqslant g\left(E_m-E_n\right) 
\end{eqnarray}
for all $m>n$ and all sufficiently large $n$.

Then
\begin{eqnarray}\label{qbcl}
S_n = E_n +o\left(\sqrt{g(E_n) \psi(\log E_n)} (\log E_n)^{3/2}\right) \quad\mbox{a.s.}
\end{eqnarray}
\end{theorem}

Relation (\ref{qbcl}) with $g(x)=x$ and $\psi(x)=x^\varepsilon$, $\varepsilon>0$,
has been obtained in Philipp [9] under
\begin{eqnarray}\label{Ph}
 P(A_i A_j) \leqslant P(A_i) P(A_j) + b_{j-i} P(A_i)
\end{eqnarray}
for all $i>j$, 
where $ \{b_n\} $ is a sequences of real numbers such that $\sum\limits_{n=1}^\infty b_n<\infty$.
When condition (\ref{var}) holds for $g(x)= Cx$, relation (\ref{qbcl}) has been derived in Petrov [8].
Note that inequalities (\ref{Ph}) imply condition (\ref{var}) with $g(x)= Cx$. 
Theorem \ref{QBCL} generalizes the mentioned results. To prove Theorem \ref{QBCL}, 
we use a modification of the methods from [8--10].

Taking $\psi(x)= (\log x)^{1+\varepsilon}$
and $g(x)=C x^{1+\gamma}$, we arrive at the next result.

\begin{corollary}
Assume that $ E_n \to\infty $ as $ n\to\infty $ and 
$ Var (S_m-S_n) \leqslant C \left(E_m-E_n\right)^{1+\gamma} $
for all $m>n$ and all sufficiently large $n$, where $ \gamma\in [0,1) $ and $C>0$.

Then
\begin{eqnarray}\label{qbcla}
S_n = E_n +o\left(E_n^{(1+\gamma)/2} (\log E_n)^{3/2}
 (\log\log E_n)^{(1+\varepsilon)/2}
\right) \quad\mbox{a.s.}
\end{eqnarray}
for all $ \varepsilon>0 $.
\end{corollary}

A verification of inequalities (\ref{var}) is the main problem for
applications of Theorem \ref{QBCL} and one need
simple conditions sufficient for (\ref{var}). Moreover, examples are
of interest for $g(x)$ increasing faster than $x$. 
The following two results yield such conditions and examples. 

\begin{theorem}\label{QBCL2}
Let $c(x)$, $x\geqslant 0$, be a positive non-increasing
continuous function such that $c(x)\to 0$ as $x\to\infty$
and $f(x)=x/c(x)$ is increasing.
Put $c_n=c(n)$ for all natural $n$.
Let $ \{b_n\} $ a sequences of real numbers such that $\sum\limits_{n=1}^\infty b_n<\infty$.
Assume that
\begin{eqnarray}\label{pr}
 P(A_i A_j) \leqslant (1+c_{j-i}) P(A_i) P(A_j) + b_{j-i} P(A_i)
\end{eqnarray}
for all $i>j$. 

Then relation (\ref{qbcl}) holds with
$g(x)=x f^{-1}(x)$, where $f^{-1}(x)$ is the inverse function to $ f(x)$. 
If inequalities (\ref{pr}) hold for all $i>j$ with $0$ instead of $c_{i-j}$, then
 relation (\ref{qbcl}) holds for $g(x)=x$.
\end{theorem}

If $ \sum\limits_{n=1}^\infty c_n <\infty $ in Theorem \ref{QBCL2},
then, replacing $b_n$ by $b_n+c_n$, we arrive at 
inequalities (\ref{pr}) with $c_{i-j}=0$ which coincide with inequalities (\ref{Ph}).

The Borel--Cantelli lemma play an important role in an analysis of
statistical properties of dynamical systems. (Cf., for example, [11-15] and the references therein.)  
We consider some non-uniformly hyperbolic (expanding) systems. 

Models of evolutions of dynamical systems are given by probability spaces and
transformations of these spaces.  
Let $T$ be an ergodic measure-preserving transformation 
of a probability space $ (X,\mathcal{B},\mu) $ and
$ \{B_n\} $ be a sequence of set such that
$B_n \in \mathcal{B}$ and $ \sum\limits_{n=1}^\infty \mu B_n =\infty $. (If the last series converges, then the first part
of the Borel--Cantelli lemma easily yields the solution of the 
below problem.) The second part of the Borel--Cantelli lemma
can give answer on the question whether
$T^n(x)$ belongs to $B_n$ infinitely often for almost every $x\in X$.
The quantitative  Borel--Cantelli lemma can also yield bounds
for numbers of visits of orbits of $x$ in $B_k$ up to time $n$.

The main problem is that events usually are not pair-wise independent for dynamical
systems. Hence, variants of the Borel--Cantelli can be useful
and our results can be applied as well.

Consider the following family of interval maps with neutral fixed points.
For $ \alpha\in (0,1) $, define $T_\alpha: (0,1] \to (0,1]$  by
\begin{equation*}
T_{\alpha}(x) =
  \begin{cases}
    x(1+2^{\alpha} x^{\alpha}), & \mbox{if}\; x \in (0,1/2], \\
    2x-1, &  \mbox{if}\; x \in (1/2,1].
  \end{cases}
\end{equation*}
It is known that $ T_{\alpha} $ preserves 
a unique probability measure $ \mu $ which is absolutely continuous with respect to Lebesgue's measure $ \lambda $.

Take $\Omega=(0,1]$, $\mathcal{F}$ being $\sigma$-field of Lebesgue's subset of $(0,1]$ and $P=\mu$. 
Applying of Theorem \ref{QBCL2} yields the following result.  

\begin{theorem}\label{IM}
Let $\{B_n\}$ be a sequence of intervals such that $ \sum\limits_{n=1}^\infty \mu B_n =\infty $
and $B_n \subset (d,1]$ for all $n$ and some $d>0$. 
Put $A_n=\{T_{\alpha}^{n} x \in B_n\}$
and define $S_n$ and $E_n$ by (\ref{snen}).

Then 
\begin{eqnarray}\label{qbclb}
S_n = E_n +o\left(E_n^{(1+\alpha)/2} (\log E_n)^{3/2}
 (\log\log E_n)^{(1+\varepsilon)/2}
\right) \quad\mu-\mbox{a.s.}
\end{eqnarray}
for all $ \varepsilon>0 $.

\end{theorem}

Theorem \ref{IM} implies the next result.

\begin{corollary}\label{SBC}
Under the conditions of Theorem \ref{IM}, sequence of intervals $\{B_n\}$
is a strong Borel--Cantelli (SBC) sequence with respect to $T_\alpha$ and $\mu$, i.e.
$$ \frac{S_n}{E_n} \to 1\quad  \mu-\mbox{a.s.}
$$
\end{corollary}

Kim [11] proved that $\{B_n\}$ is a SBC sequence of intervals provided either
$B_{n+1}\subset B_n$ for all $n$ and $0 \notin \cup_n \bar{B_n}$, 
or $\alpha<(3-\sqrt{5})/2$ and $B_n$ lie in $(d,1]$. Moreover, 
Kim [11] also showed that $\{B_n\}$, $B_n=[0,n^{1/(\alpha-1)})$, 
is not the SBC sequence despite $\sum\limits_n \mu B_n$ diverges.
This implies that the result of Corollary \ref{SBC} 
can fail for intervals with left endpoint at zero.  

Gou\"{e}zel [12] has proved that for almost every $x$, $T_{\alpha}^{n} x$ belongs to $B_n$ 
infinitely often provided the intervals $B_n$ satisfy to
$ \sum\limits_{n=1}^\infty \lambda B_n =\infty $.

For $\{B_n\}$ containing in $(d,1]$, Theorem \ref{IM} improves the mentioned results
from Kim [11] and Gou\"{e}zel [12]. Note that measures $\mu$ and $\lambda$ are equivalent
for $\{B_n\}$ separated from zero. 

\section*{2. Proofs}

\begin{proof}[Proof of Theorem \ref{QBCL}]
For $m>n$, denote
\begin{equation}\label{nt}
S(n,m) = S_m-S_n, \quad E(n,m) = E_m-E_n,\quad \tilde{S}(n,m)= S(n,m)-E(n,m).
\end{equation}

Put $N_u = \max\{ n: E_{n} <u\}$ for every integer $u \geqslant 0$. 
Let $r$ and $s$ be integer numbers with $ r\geqslant 1 $ and $0\leqslant s \leqslant r$. Then 
for every fixed $s$, we have
$\bigcup\limits_{t=0}^{2^{r-s}-1} (t 2^s, (t+1) 2^s] = (0,2^r]$
and
\begin{eqnarray}\label{10}
\sum\limits_{t=0}^{2^{r-s}-1} E\left(N_{t 2^s},N_{(t+1) 2^s}\right)
\leqslant E_{N_{2^r}} < 2^r.
\end{eqnarray}
Put
\begin{eqnarray*}
T_r = \sum\limits_{s=0}^r \sum\limits_{t=0}^{2^{r-s}-1} \left(
\tilde{S}\left(N_{t 2^s},N_{(t+1) 2^s}\right)\right)^2.
\end{eqnarray*}
Taking into account inequalities (\ref{var}) and (\ref{10}) and the monotonicity
of $ g(x)/x $, we get
\begin{eqnarray*}
&&
E T_r =\sum\limits_{s=0}^r \sum\limits_{t=0}^{2^{r-s}-1} 
Var S\left(N_{t 2^s},N_{(t+1) 2^s}\right) 
\leqslant \sum\limits_{s=0}^r \sum\limits_{t=0}^{2^{r-s}-1} 
g\left(E\left(N_{t 2^s},N_{(t+1) 2^s}\right) \right)
\\ &&
\leqslant \sum\limits_{s=0}^r \sum\limits_{t=0}^{2^{r-s}-1} 
\frac{g\left(E\left(N_{t 2^s},N_{(t+1) 2^s}\right)\right) }{E\left(N_{t 2^s},N_{(t+1) 2^s}\right)}E\left(N_{t 2^s},N_{(t+1) 2^s}\right)
\\ &&
\leqslant \sum\limits_{s=0}^r \sum\limits_{t=0}^{2^{r-s}-1} 
\frac{g\left(2^r \right) }{2^r} E\left(N_{t 2^s},N_{(t+1) 2^s}\right)
\leqslant (r+1) g\left(2^r \right). 
\end{eqnarray*}

Note that $\sum_{n=1}^\infty 1/(n \psi (c n))<\infty$ for every fixed $c>0$.
Take $c=(\log 2)/4$.

Hence,
\begin{eqnarray*}
\sum\limits_{r=1}^\infty \frac{ET_r}{r^2 g(2^r) \psi(c r)} \leqslant 2
\sum\limits_{r=1}^\infty \frac{1}{r \psi(c r)}
< \infty.
\end{eqnarray*}
This implies that
\begin{eqnarray*}
\sum\limits_{r=1}^\infty \frac{T_r}{r^2 g(2^r) \psi(c r)} < \infty\quad\mbox{and}\quad
\frac{T_r}{r^2 g(2^r) \psi(c r)}\to 0 \quad\mbox{as}\quad r\to\infty
\quad\mbox{a.s.}
\end{eqnarray*}

Take integer $k$ such that $2^{r-1}<k\leqslant 2^r$. By the Cauchy--Bunyakovskii inequality,
we get from the last relation that
\begin{eqnarray*}
&&
(S_{N_k}-E_{N_k})^2 = \left( \sum\limits_{j=0}^{r-1} 
\tilde{S}(N_{2^{j-1}}, N_{2^j})+ \tilde{S}(N_{2^{r-1}},N_{k}) \right)^2
\\ && 
\leqslant (r+1) \left( \sum\limits_{j=0}^{r-1}  \left(
\tilde{S}(N_{2^{j-1}}, N_{2^j}) \right)^2+\left( \tilde{S}(N_{2^{r-1}},N_{k}) \right)^2\right)
\\ &&
\leqslant (r+1) T_r = o\left(r^3 g(2^r) \psi(c r)\right) \quad\mbox{as}\quad r\to\infty
\quad\mbox{a.s.}
\end{eqnarray*}
We have $2^{r-1} \leqslant k-1 \leqslant E_{{N_k}+1}-1 
=E_{N_{k}}+P(A_{N_{k}+1})-1 \leqslant E_{N_k}$.
Hence, $g(2^r) \leqslant g(2  E_{N_k}) \leqslant 2^{2-\delta}g(E_{N_k})$
and $\psi(c r)\leqslant \psi(c (\log E_{N_k}/\log 2+1)) \leqslant \psi((\log E_{N_k})/2)$
for all sufficiently large $r$. 
This yields relation (\ref{qbcl}) for $n=N_k$.

For $n$ with $N_{k} \leqslant n < N_{k+1}$, we have
$S_{N_{k}} \leqslant S_n < S_{N_{k+1}}$, 
$E_{N_{k}} \leqslant E_n < E_{N_{k+1}}$
and $ E_{N_{k+1}}<k+1 \leqslant E_{N_k}+2$. 
It follows that $g(E_{N_{k+1}}) \leqslant g(E_n+2) \leqslant g(E_n)((E_n+2)/E_n))^{2-\delta}$
and $ \psi((\log E_{N_{k+1}})/2)\leqslant  \psi((\log (E_{n}+2))/2) \leqslant  \psi(\log E_{n})$
for all sufficiently large $k$. Then for every $ \varepsilon>0$, the inequalities
\begin{eqnarray*}
&& \hspace*{-\parindent}
S_n-E_n \!\leqslant\! S_{N_{k+1}} - E_{N_{k}} \!\leqslant\!  S_{N_{k+1}} - E_{N_{k+1}} + 2
\!\leqslant\! \varepsilon \sqrt{g(E_{N_{k+1}}) \psi((\log E_{N_{k+1}})/2)} (\log E_{N_{k+1}})^{3/2}
\\ && \hspace*{-\parindent}
\leqslant (1+\varepsilon) \varepsilon \sqrt{g(E_{n}) \psi(\log E_{n})} (\log E_{n})^{3/2}
\quad\mbox{a.s.}
\end{eqnarray*}
hold for all sufficiently large $n$. 
For every $ \varepsilon>0$, we also have
\begin{eqnarray*}
&&
S_n-E_n \geqslant S_{N_{k}} - E_{N_{k+1}} \geqslant  S_{N_{k}} - E_{N_{k}} - 2
\geqslant -\varepsilon \sqrt{g(E_{N_{k}}) \psi(\log E_{N_{k}})} (\log E_{N_{k}})^{3/2}
\\ &&
\geqslant -\varepsilon \sqrt{g(E_{n}) \psi(\log E_{n})} (\log E_{n})^{3/2}
\quad\mbox{a.s.}
\end{eqnarray*}
for all sufficiently large $n$. Hence, relation (\ref{QBCL}) follows.
\end{proof}

\begin{proof}[Proof of Theorem \ref{QBCL2}]
Check that conditions (\ref{pr}) imply inequalities (\ref{var}).
We use notations (\ref{nt}).

Let $k_{n,m}$ be the solution of the equation $ \frac{x}{c(x)} =  E(n,m)$.
By (\ref{pr}), we have
\begin{eqnarray*}
&&
Var (S_m-S_n) = E(n,m)- \sum\limits_{i=n+1}^m P^2(A_i) +
2 \sum\limits_{n<i<j\leqslant m} \left(P(A_i A_j) - P(A_i) P(A_j)\right)
\\ &&
\leqslant E(n,m)+ 2 \sum\limits_{n<i<j\leqslant m} c_{j-i} P(A_i) P(A_j)+
2 \sum\limits_{n<i<j\leqslant m} b_{j-i} P(A_i).
\end{eqnarray*}
For the last term, we have the following bound
\begin{eqnarray*}
 &&
 \sum\limits_{n<i<j\leqslant m} b_{j-i} P(A_i)
\leqslant \sum\limits_{i=n+1}^m  P(A_i) \sum\limits_{j=i+1}^m  b_{j-i}  
\leqslant E(n,m) \sum\limits_{j=1}^\infty  b_{j}.
\end{eqnarray*}
For the middle term, we get
\begin{eqnarray*}
 &&
\sum\limits_{n<i<j\leqslant m} c_{j-i} P(A_i) P(A_j) 
\leqslant
\left( \sum\limits_{n<i<j\leqslant m, j-i\leqslant k_{n,m}} 
 +  \sum\limits_{n<i<j\leqslant m, j-i\geqslant k_{n,m}} \right) c_{j-i} P(A_i) P(A_j)
 \\ &&
\leqslant  \sum\limits_{i=n+1}^m \sum\limits_{j=i+1}^{[k_{n,m}]} c_{j-i} P(A_i) P(A_j)
 +  c(k_{n,m}) \sum\limits_{n<i<j\leqslant m, j-i\geqslant k_{n,m}} P(A_i) P(A_j)
\\ && 
\leqslant k_{n,m} \sup\limits_n \{c_{n}\}
\sum\limits_{i=n+1}^m  P(A_i)  
 +  c(k_{n,m}) \left(E(n,m)\right)^2 
\\ && 
 \leqslant C \left(E(n,m) k_{n,m}  +  c(k_{n,m}) \left(E(n,m)\right)^2 \right)
=2 C k_{n,m} E(n,m) = 2 C g\left(E(n,m)\right).
\end{eqnarray*}
The above bounds imply relation (\ref{var}) and Theorem \ref{QBCL2} follows.
\end{proof}

\begin{proof}[Proof of Theorem \ref{IM}.]
Assume first that $B_n\subset (1/2,1]$ for all $n$.
From Gou\"{e}zel [12], we borrow bound (1.3) as follows
\begin{eqnarray*}
\left| \mu(T_\alpha^{-i} B_i \cap T_\alpha^{-j} B_j) -
(1+\mathbf{c}_{i-j}) \mu(B_i) \mu(B_j)\right| \leqslant \frac{C \mu(B_j)}{(j-i)^\beta},
\end{eqnarray*}
where $\beta=1/\alpha$ and $\mathbf{c}_n = c n^{1-\beta}(1+o(1))$ as $n\to\infty$
for some non-zero constant $c$.
Then $C_1 = \sup\limits_n |\mathbf{c}_n| n^{\beta-1}<\infty$. 
Take $c(x)= C_1 x^{1-\beta}$ and $b_n= C n^{-\beta}$. 
If $\alpha\geqslant 1/2 $, then the conditions of Theorem \ref{QBCL2} hold for $g(x)= C_2 x^{1+\alpha}$. 
By Corollary 1, the result follows.
For $ \alpha<1/2 $, we get $\sum\limits_n c_n<\infty$ and
Theorem \ref{QBCL2} with $g(x)=C_3 x$
implies the result. 

We further use the following agreements. We write that relation
(\ref{qbclb}) holds for a sequence of measurable sets $\{B_n\}$ if it 
holds with $S_n$ and $E_n$ defined by (\ref{snen})
for $A_n=\{T_{\alpha}^{n} x \in B_n\}$. We
define $S_n^k$ and $E_n^k$ replacing $\{B_n\}$ by $\{B_n^k\}$ for every
fixed natural k.

Note that $B_n$ can be arbitrary measurable sets in the next result.

\begin{lemma}\label{l1}
If relation (\ref{qbclb}) holds for a sequence of sets $\{B_n\}$, then
(\ref{qbclb}) holds for  $\{T_\alpha^{-1} B_n\}$.
\end{lemma}

\begin{proof}[Proof of Lemma \ref{l1}.]
Put $B_n^1= T_\alpha^{-1} B_n$ and $A_n^1=\{x: T_\alpha^n x \in B_n^1\}$ for all n. 
Then we get
\begin{eqnarray*}
&&
S_n^1 = \sum\limits_{k=1}^n \mathbb{I} (A_k^1) =
 \sum\limits_{k=1}^n \mathbb{I} (\{x: x \in T_\alpha^{-k}(B_n^1)  \})
 = \sum\limits_{k=1}^n \mathbb{I} (\{x: x \in T_\alpha^{-(k+1)}(B_n)  \})
\\ && 
= \sum\limits_{k=1}^{n+1} \mathbb{I} (\{x: x \in T_\alpha^{-k}(B_n)  \})
 -  \mathbb{I} (\{x: x \in T_\alpha^{-1}(B_n)  \})
 = S_{n+1} - \mathbb{I} (\{x: x \in T_\alpha^{-1}(B_n)  \}).
\end{eqnarray*}
This follows that $S_{n+1}-1 \leqslant S_n^1 \leqslant S_{n+1}$ and
$E_{n+1}-1 \leqslant E_n^1=E S_n^1 \leqslant E_{n+1}$. Hence,
\begin{eqnarray*}
S_n^1 = E_{n+1} + o\left(f(E_{n+1})\right)  = E_n^1+ o\left(f(E_{n}^1)\right) \quad \mu-\mbox{a.s.},
\end{eqnarray*}
where
$$f(x)= x^{(1+\alpha)/2} (\log x)^{3/2}  (\log\log x)^{(1+\varepsilon)/2}$$ for $x>e^e$.
 The result follows.
\end{proof}

Put $a_0=1/2$ and $a_k = T^{-1}\vert_{(0,1/2]} (a_{k-1})$ for $k\geqslant 1$. 

Assume now that $B_n\subset (a_1,a_0]$ for all $n$. Then the intervals
$B^2_n=T_\alpha (B_n)$ lie in $(1/2,1]$ and we get
\begin{eqnarray*}
&&
S_n^2= E_n^2+ o\left(f(E_n^2) \right) \quad \mu-\mbox{a.s.}
\end{eqnarray*}
We have $B_n = B_n^3 \backslash B_n^4$, where 
$B_n^3 = T_\alpha^{-1}(B^2_n)$ and $B_n^4= (T_\alpha^{-1}(B^2_n)\cap (1/2,1])$.
By Lemma~\ref{l1} for $B_n=B_n^2$, we have
\begin{eqnarray*}
S_n^3= E_n^3+ o\left(f(E_n^3)\right) \quad \mu-\mbox{a.s.}
\end{eqnarray*}
For the intervals $B_n^4$, we get
\begin{eqnarray*}
S_n^4= E_n^4+ o\left(f(E_n^4)\right) \quad \mu-\mbox{a.s.}
\end{eqnarray*}
Taking into account that $S_n=S_n^3-S_n^4$ and $E_n=E_n^3-E_n^4$, we arrive at
\begin{eqnarray*}
S_n = E_n + o\left(f(E_n^3))\right)+ o\left(f(E_n^4)\right)\quad \mu-\mbox{a.s.}
\end{eqnarray*}
The two last terms are of the same order as 
$o\left(f(E_n)
\right)$ since 
$\mu B_n^4 \leqslant c \mu B_n^3$
for some $c\in(0,1)$ which gives $E_n \geqslant (1-c) E_n^3 \geqslant (1-c) E_n^4$.

Suppose now that $B_n\subset (a_1,1]$ for all $n$. Then $B_n=B_n^5 \cup B_6$, where
$B_n^5\subset (a_1,a_0]$ and $B_n\subset (a_0,1]$. Applying of relation (\ref{qbclb})
to $\{B_n^5\}$ and $\{B_n^6\}$ yields the result in the case under consideration.

For $B_n\subset (a_k,1]$ with $k\geqslant 2$, we obtain the result by induction.
\end{proof}

\bigskip
\noindent
{\bf References}
{\parindent0mm
{\begin{itemize}
\footnotesize
\item[{[1]}]
Chung K.L., Erd\H{o}s P., 1952. On the application of 
the Borel-Cantelli lemma.Trans. Amer. Math. Soc. 72, 179--186.
\item[{[2]}]
Erd\H{o}s P., R\'{e}nyi A., 1959. On Cantor's series with 
convergent $\sum 1/q$, Ann. Univ. Sci. Budapest 
Sect. Math. 2, 93--109.
\item[{[3]}]
Spitzer F., 1964. Principles of random walk. Van Nostrand, Princeton.
\item[{[4]}]
M\'{o}ri T.F., Sz\'{e}kely  G.J., 1983.
On the Erd\H{o}s--R\'{e}nyi generalization of the
Borel--Cantelli lemma. Studia Sci. Math. Hungar. 18, 173-182.
\item[{[5]}]
Petrov V.V., 2002. A note on the Borel--Cantelli lemma,
Statist.~Probab.~Lett. 58, 283--2866.
\item[{[6]}]
Frolov A.N., 2012. Bounds for probabilities of unions of events 
and the Borel--Cantelli lemma. Statist. Probab. Lett. 82, 2189--2197.
\item[{[7]}]
Frolov A.N., 2015. On lower and upper bounds for probabilities of unions and the Borel--Cantelli lemma.
Studia Sci. Math. Hungarica. 52 (1), 102--128.
\item[{[8]}]
Petrov, V.V. The Growth of Sums of Indicators of Events. J Math Sci 128, 2578--2580 (2005). Translated from Zapiski Nauchnykh Seminarov POMI, Vol. 298, 2003, pp. 150--154.
\item[{[9]}]
Phillipp, W. Some metrical theorems in number theory. Pacific J. Math. 20 (1967), 109--127.
\item[{[10]}]
Schmidt, W. Metrical theorems on fractional parts of sequences. Trans. Amer. Math. Soc. 110 (1964),
493--518.
\item[{[11]}]
Kim, D. The dynamical Borel--Cantelli lemma for interval maps. Discrete Contin. Dyn. Syst. 17(4) (2007), 891--900.
\item[{[12]}]
Gou\"{e}zel S. A Borel--Cantelli lemma for intermittent interval maps. Nonlinearity 20(6) (2007), 1491--1497.
\item[{[13]}]
Gupta C., Nicol M. and Ott W. A Borel--Cantelli lemma for 
non-uniformly expanding dynamical systems.
Nonlinearity 23(8) (2010), 1991--2008.
\item[{[14]}]
Haydn N., Nicol M., Persson T., Vaienti S. A note on Borel-Cantelli lemmas for nonuniformly
hyperbolic dynamical systems. Ergodic Theory Dynam. Systems 33 (2013), 2, 475--498.
\item[{[15]}]
Luzia N., Borel-Cantelli lemma and its applications. Trans. Amer. Math. Soc. 366 (2014), no. 1, 547--560.
\end{itemize}
}
}

\end{document}